\documentclass{amsart}

\usepackage{amsmath}
\usepackage{amssymb}
\usepackage{amsthm}
\usepackage{stmaryrd}
\usepackage[hidelinks, colorlinks=true, linkcolor=blue, citecolor=blue]{hyperref}
\usepackage{graphicx}
\usepackage[dvipsnames]{xcolor}
\usepackage{enumitem}
\usepackage{mathtools}

\theoremstyle{plain}
\newtheorem{theorem}{Theorem}
\newtheorem{corollary}[theorem]{Corollary}
\newtheorem{lemma}[theorem]{Lemma}
\newtheorem{proposition}[theorem]{Proposition}

\theoremstyle{remark}
\newtheorem{remark}[theorem]{Remark}

\numberwithin{equation}{section}
\numberwithin{theorem}{section}

\newcommand{\R}{\mathbb{R}}
\newcommand{\N}{\mathbb{N}}
\newcommand{\Domain}{\Omega}
\newcommand{\trial}{\mathbb{U}}
\newcommand{\test}{\mathbb{V}}
\newcommand{\data}{\mathbb{V}^*}

\newcommand{\fsin}{f_{k}^{\mathtt{s}}}
\newcommand{\fcos}{f_{k}^{\mathtt{c}}}
\newcommand{\fexp}{f_{k}^{\mathtt{e}}}
\newcommand{\ftsin}{\widetilde f_{k}^{\mathtt{s}}}
\newcommand{\ftcos}{\widetilde f_{k}^{\mathtt{c}}}
\newcommand{\sinc}{\operatorname{sinc}}

\newcommand{\dt}{\mathrm{d}t}

\begin{document}
	
	\title[Resonant solutions and (in)stability of the linear wave equation]{Resonant solutions and (in)stability\\ of the linear wave equation}
	
	\author[G. Sangalli]{Giancarlo Sangalli}
	\address{Università degli Studi di Pavia, Dipartimento di Matematica,  27100 Pavia, Italy \newline
    \hspace*{1.5em}Istituto di Matematica Applicata e Tecnologie Informatiche, ``E. Magenes" -- CNR, 27100 Pavia, Italy.}
	\email{giancarlo.sangalli@unipv.it}
	
	\author[D.~Terazzi]{Davide Terazzi}
	\address{Università di Milano Bicocca, Dipartimento di Matematica e Applicazioni, 20125 Milano, Italy}
	\email{d.terazzi@campus.unimib.it}
	
	\author[P.~Zanotti]{Pietro Zanotti}
	\address{Universit\`{a} degli Studi di Milano, Dipartimento di Matematica, 20131, Milano, Italy}
	\email{pietro.zanotti@unimi.it}
	
	
	
	%
	%
	%
	
	\begin{abstract}
	We revise the analysis of the acoustic wave equation, addressing the question whether the classical well-posedness implies the existence of an isomorphism between prescribed solution and data spaces. This question is of interest for the design and the analysis of discretization methods. Expanding on existing results, we point out that established choices of solution and data space in terms of classical Bochner spaces must be expected to be incompatible with the existence of such an isomorphism, because of resonant waves. We formulate this observation in the language of the so-called inf-sup theory, with the help of an eigenfunction expansion, which reduces the original partial differential equation to a system of ordinary differential equations. We further verify that an isomorphism can be established, for each equation in the system, upon equipping the data space with a suitable resonance-aware norm. In the appendix, we extend our results to other time-dependent linear PDEs.
	\end{abstract} 
	
	\maketitle
	
\section{Introduction}
\label{S:introduction}

Wave equations model a number of oscillatory phenomena of interest in physics and engineering, including acoustics, electromagnetism, and structural vibrations. The construction and the analysis of approximation methods for such equations have been the subject of intensive studies over the years but still pose substantial challenges. One main difficulty originates in the analysis of the equations and critically distinguishes them from elliptic and parabolic ones. 

To introduce the subject of our interest, consider the following abstract linear equation: for $f \in \mathbb{D}\subseteq\data$, find $u \in \trial$ such that
\begin{equation}
\label{eq:abstract-problem}
\mathcal{L} u = f \quad \text{in $\data$},
\end{equation}
where $\trial$ and $\test$ are Hilbert spaces and $\mathcal{L}: \trial \to \data$ is a bounded linear operator. The  problem is  well-posed in the classical (Hadamard) sense when the solution depends continuously on the right hand side, that is 
\begin{equation}
\label{eq:hadamard-well-posedness}
 \| u \|_{\trial} \lesssim \|f\|_{\mathbb{D}},
\end{equation} 
for a given norm $\|\cdot\|_{\mathbb{D}}$, which in general may be stronger than  $\|\cdot\|_{\data}$. As the continuity of $\mathcal{L}$ implies the lower bound $\|f\|_{\data} \lesssim \| u \|_{\trial}$, then $\mathcal{L}$ establishes an isomorphism in the particular case when $\mathbb D = \data$ , that is
\begin{equation}
\label{eq:isomorphism}
  \| u \|_{\trial} \eqsim \|f\|_{\data}.
\end{equation} 
This equivalence plays a pivotal role in the design and analysis of approximation methods. Establishing a counterpart of \eqref{eq:isomorphism} for a discretization of \eqref{eq:abstract-problem} lays the groundwork for sharp stability and a priori error estimates \cite{Babuska:70, Brezzi:74}. Other relevant applications are found in the a posteriori error analysis \cite{Verfurth:13}, the development of preconditioning techniques \cite{Hiptmair:06, Mardal.Winther:11}, the theory of optimal control \cite{Loscher.Steinbach} and the analysis of adaptive techniques \cite{Feischl:22}.

It is nowadays well-established that,  for a variety of elliptic, parabolic and saddle point problems, using Sobolev- and Bochner-type spaces,  well-posedness in the stronger sense of the equivalence \eqref{eq:isomorphism} holds, and can be characterized by means of the  Banach-Ne\v{c}as-Babu\v{s}ka theorem, see \cite{Ern.Guermond:21-II, Ern.Guermond:21-III} for an overview (see also \cite{Kreuzer.Zanotti:24a} for mixed elliptic-parabolic systems). In contrast, the application of the same technique to the acoustic wave equation is notoriously critical. In particular, a remarkable counter-example by Zank \cite{Zank:20} shows that the equivalence \eqref{eq:isomorphism} fails for a typical definition of $\trial$ and $\test$ in terms of Bochner-like spaces. 

As a first contribution, we elucidate the mechanism underlying the counter-example by Zank. We consider a simple initial-boundary value problem (IBVP) for the acoustic wave equation. We propose a standard definition of $\trial$ and $\test$ and we observe that resonant waves prevent from establishing the desired equivalence. Indeed, those waves verify that, for the bilinear form $\mathcal{B}(\cdot, \cdot) := \left\langle \mathcal{L}\cdot, \cdot\right\rangle  $ induced by $\mathcal{L}$, the so-called inf-sup condition from the Banach-Ne\v{c}as-Babu\v{s}ka theorem fails, i.e.
\begin{equation}
\label{eq:inf-sup-failure}
		\inf_{\substack{\widetilde u \in \trial}}\,\,\sup _{\substack{v \in \test}} \,\frac{\mathcal{B}(\widetilde u,v)}{\|\widetilde u\|_\trial\|v\|_\test} = 0. 
\end{equation}
Our proof builds upon an eigenfunction expansion, which reduces the original IBVP to an infinite system of initial value problems (IVPs). Although our result hinges on a specific functional setting, the underlying principle is fairly general, suggesting that the same issue arises for different definitions of $\trial$ and $\test$ in terms of Bochner-like spaces as well as for other equations modeling oscillatory phenomena.  

Devising a convenient functional setting for the analysis of wave equations is a long-standing problem. On the one hand, various results address the well-posedness of wave equations by providing bounds of the type \eqref{eq:hadamard-well-posedness} (with $\|\cdot\|_{\mathbb{D}}$  stronger than  $\|\cdot\|_{\data}$) including the classical references \cite{Dautray.Lions:92, Evans:10, Ladyzhenskaya:85, Lions.Magenes:72} as well as more recent contributions \cite{Dong.Gergoulis.Mascotto.Wang:25, Dong.Mascotto.Wang:25, Ferrari.Perugia:25, Ferrari.Perugia.Zampa:25}.
On the other hand, to our best knowledge, approaches that aim at establishing the stronger equivalence \eqref{eq:isomorphism} have been pursued along two main lines:
\begin{enumerate}[label=(\roman*)]
	\item Impose that the equation establishes an isomorphism, by augmenting the norm on one of the spaces with a term involving the operator $\mathcal{L}$. For instance, given the norm $\|\cdot\|_{\test}$, equip $\trial$ with $\|\mathcal{L}\cdot \|_{\data}$, see \cite{Bignardi.Moiola:25, Fuhrer.Gonzalez.Karkulik:25, Steinbach.Zank:22, Zank:20} and \cite{Henning.Palitta.Simoncini.Urban:22} for the same argument, upon exchanging the roles of $\trial$ and $\test$.
	\vspace{5pt}
	\item Modify the equation by acting on the solution with a suitable operator, and show that the transformed problem establishes an isomorphism \cite{Bignardi.Moiola:25, Hoonhout.Loscher.UrzuaTorres:25}.
\end{enumerate} 
In (i), at least one of the spaces is characterized `implicitly' via the operator $\mathcal{L}$ and not only in terms of classical regularity. In (ii), both the regularity requirements and the invertibility of the transformation must be carefully addressed.

As a second contribution,  we aim at an explicit (i.e. independent of $\mathcal{L}$) characterization of a combination of spaces and norms realizing the equivalence \eqref{eq:isomorphism}. We consider a single IVP resulting from the eigenfunction expansion of the above-mentioned IBVP. In this simplified scenario, we propose a standard definition of $\|\cdot\|_\trial$ and derive two alternative but equivalent explicit expressions for $\|\cdot\|_{\data}$. In particular, we distinguish between a spectral approach via Fourier expansion, and a temporal approach based on a suitable energy balance law. The resulting norm is resonance-aware, i.e., it correctly accounts for data that excite a resonance.  Our findings suggest that it might be necessary moving beyond traditional Bochner-type spaces.

\subsection*{Outline}
Section \ref{S:model-problem-statement} introduces the model IBVP, the setting for our analysis and the eigenfunction expansion. Section \ref{S:Zank-counter-example} elucidates the role played by resonant waves in establishing \eqref{eq:inf-sup-failure}. Section \ref{S:resonance-aware-norms} introduces a new resonance-aware setting for the analysis. Finally, Appendix \ref{A:extension-other-IBVPs} provides an extension of our results to other initial-boundary value problems.

\subsection*{Notation} We use standard notation and norms for Lebesgue and Sobolev spaces $L^p(D)$, $H^s_0(D)$ and $H^s(D)$ with $p \in [1,\infty]$, $s\in\N$ and $D \subseteq \R^d$, $d\in \N$. We write $H^{-s}(D)$ for the dual space of $H^{s}_0(D)$. For a Banach space $\mathbb X$, we denote by $H^s(\mathbb X)$, $s\in\N$, and $L^\infty(\mathbb X)$ the spaces of $H^s$ and $L^\infty$ functions,  mapping the time interval $[0, T]$ into $\mathbb X$. We equip those spaces with the Bochner norms $\|\cdot\|_{H^s(\mathbb X)}$ and $\|\cdot\|_{L^\infty(\mathbb X)}$, respectively. To alleviate the notation, we abbreviate, e.g., $H^s(L^2(D))$ to $H^s(L^2)$, whenever there is no ambiguity on $D$. For $x, y \in \R$, we write $x \eqsim y$ or equivalently $x\lesssim y \lesssim x$ when there are constants $0 < C_1 \leq C_2$ such that $C_1 x \leq y \leq C_2 x$. The dependence of the hidden constants on relevant quantities is addressed case by case. Finally, we denote by $\sinc(x)$ the unnormalized sinc function
\begin{equation*}
\sinc(x) := \frac{\sin(x)}{x}
\quad \text{for $x\neq 0$} \qquad \text{and} \qquad \sinc(x) := 1 \quad \text{for $x=0$}.
\end{equation*}

\section{Model problem statement}
\label{S:model-problem-statement}
This section introduces a model IBVP and a framework for its analysis. We consider the simplest possible model, in order to better illustrate the essence of our results. We comment on more general boundary conditions, initial conditions and other model problems in Remark~\ref{rem:BoundaryConditions} and \ref{rem:nonzero-initial-values} and Appendix \ref{A:extension-other-IBVPs}, respectively. 

Let $\Domain\subset\R^d$, $d\in\N$, be an open, bounded and connected set, whose boundary can be locally represented as the graph of a Lipschitz-continuous function. Denote by $T>0$ the time horizon. We consider the following IBVP for the acoustic wave equation:
\begin{equation}
	\label{eq:model-problem}
		\begin{aligned}
			 \frac{1}{c^2}\partial_{tt}u - \Delta u & = f && \text{in } \phantom{\partial } \Omega \times (0,T) \\
			u & = 0 && \text{in } \partial \Omega \times (0,T) \\
			u(\cdot, 0) & = 0 && \text{in } \Omega \\
			\partial_t u(\cdot, 0) & = 0 && \text{in } \Omega,
		\end{aligned}
\end{equation}
where $f: \Omega \times (0,T) \to \R$ denotes the source term and $c > 0$ is the wave speed.

\subsection{Functional setting}
\label{SS:functional-setting}
We aim at recasting the model problem into the abstract form \eqref{eq:abstract-problem}, i.e. we are interested in a weak formulation of the IBVP \eqref{eq:model-problem}. 
Motivated by classical well-posedness bounds \cite[Chapter XVIII, Section $5$]{Dautray.Lions:92}, \cite[Section $7.2$]{Evans:10}, \cite[Theorem 3.2]{Ladyzhenskaya:85}, \cite[Theorem 8.1]{Lions.Magenes:72} for the acoustic wave equation, we search for a solution of \eqref{eq:model-problem} in the \emph{trial} space
\begin{equation}
	\label{eq:WaveTrialSpace}
	\trial := \{\widetilde{u} \in L^2(H^1_0(\Omega)) \cap H^2(H^{-1}(\Omega)) \mid \widetilde u(0) = 0, \: \partial_t \widetilde u(0) = 0 \}
\end{equation}
equipped with the norm
\begin{equation}
	\label{eq:WaveTrialNorm}
	\|\cdot\|^2_\trial:=\|\cdot\|^2_{L^2(H^1_0)} + \frac{1}{c^4}\|\partial_{tt} \cdot\|^2_{L^2(H^{-1})}.
\end{equation}
Then, the first equation in \eqref{eq:model-problem} suggests to consider the \emph{test} and the \emph{data} space as
\begin{equation*}
\label{eq:WaveTestDataSpace}
\test := L^2(H^{1}_0(\Omega)) \qquad \text{and} \qquad \data := L^2(H^{-1}(\Omega))	
\end{equation*}
equipped with the norms 
\begin{equation}
\label{eq:WaveTestDataNorm}
\|\cdot\|_{\test} := \|\cdot\|_{L^2(H^1_0)}  
\qquad \text{and} \qquad
\|\cdot\|_{\data} := \|\cdot\|_{L^2(H^{-1})}.
\end{equation}

\begin{remark}[Embedding theorems]
\label{rem:embedding-theorem}
According to \cite[Theorems 2.3, 3.1]{Lions.Magenes:72}, there are continuous embeddings of $L^2(H^1_0(\Omega)) \cap H^2(H^{-1}(\Omega))$ into the spaces $H^1(L^2(\Omega))$, $L^\infty(\mathbb H^{1/2}(\Omega))$ and $W^{1,\infty}(\mathbb H^{-1/2}(\Omega))$, cf. \eqref{eq:SpectralRepresentationOfBochnerSpaces}-\eqref{eq:InterpolationSpaces} below. The embeddings are sharp, in that the differentiability and integrability indices of the latter spaces cannot be made larger.
\end{remark}

With the above spaces at hand, we consider the following weak formulation of \eqref{eq:model-problem}: find $u \in \trial$ such that
\begin{equation}
\label{eq:model-problem-weak}
\mathcal{L} u := \frac{1}{c^2}\partial_{tt} u - \Delta u = f \quad \text{in $\data$}
\end{equation}
where the differential operators are meant in weak form. We are interested in the question whether $\mathcal{L}$ establishes an isomorphism between $\trial$ and $\data$. Since $\mathcal{L}$ is bounded, this is equivalent to asking whether \eqref{eq:isomorphism} holds.

\begin{remark}[Data space]
\label{rem:data-space}
In the classical well-posedness framework established in \cite{Dautray.Lions:92, Evans:10, Ladyzhenskaya:85, Lions.Magenes:72}, the source term $f$ is required to possess higher regularity than in \eqref{eq:model-problem-weak}, namely $f \in \mathbb D := L^2(L^2(\Omega))$, under which condition \eqref{eq:hadamard-well-posedness} is satisfied. However, we are  led to enlarge the admissible data space from $\mathbb D$ to $\data$ in order to guarantee the boundedness of $\mathcal{L}$. The remaining question is whether $\data$ is too large: this is further explored in the following sections.
\end{remark}

\begin{remark}[Dimensional analysis]
\label{rem:dimensional-analysis}
Another justification for our definition of the data space is obtained by dimensional analysis. Indeed, for $d = 3$, the norm $\|u\|_\trial$ of the solution and the norm $\|f\|_{\data}$ of the source term in \eqref{eq:model-problem-weak} have the same physical dimension.  The same observation motivates also the scaling by the constant $c$ in \eqref{eq:WaveTrialNorm}.
\end{remark}

\begin{remark}[Operator components]
\label{rem:OperatorComponents}
    Exploiting the definition of $\|\cdot\|_\trial$ and $\|\cdot\|_{\data}$ above, we infer that \eqref{eq:isomorphism} holds true in this setting if and only if the norms
    \begin{equation*}
        \|(c^{-2}\partial_{tt} - \Delta) \cdot\|_{L^2(H^{-1})}
        \qquad \text{and} \qquad
        \Big(\|c^{-2}\partial_{tt} \cdot\|^2_{L^2(H^{-1})} + \|-\Delta \cdot\|^2_{L^2(H^{-1})}\Big)^{\frac{1}{2}}
    \end{equation*}
    are equivalent on $\trial$. Thus, we are led to ask whether we can separately control the space and time differential operators in the definition of $\mathcal{L}$ with respect to the data norm. For the heat equation, i.e. in the parabolic case, the answer is known to be affirmative, see e.g. \cite[Section 65]{Ern.Guermond:21-III} and Section~\ref{SS:Heat equation} below. 
\end{remark}

\subsection{Function spaces} 
We adopt an eigenfunction-based representation of the Bochner spaces used in this paper, which provides a convenient spectral interpretation of the associated norms. Denote by $(e_k)_{k=1}^{+\infty} \subseteq H^1_0(\Omega)$ and $(\lambda_k)_{k=1}^{+\infty} \subseteq \R$ the eigenfunctions and the eigenvalues of the  Laplacian $-\Delta: H^1_0(\Omega) \to H^{-1}(\Omega)$, with the normalization $\|e_k\|_{L^2(\Omega)} = 1$ for $k \geq 1$. Recall that the eigenvalues are such that 
\begin{equation}
\label{eq:eigenvalues}
0 < \lambda_k \nearrow +\infty.
\end{equation}
For $s\in[0,1]$, define
\begin{equation}
    \label{eq:SpectralRepresentationOfBochnerSpaces}
     L^2(\mathbb H^s(\Omega)) := \left\{w\in L^2(L^2(\Omega)) \mid \|w\|_{L^2(\mathbb H^s)}^2 <+\infty \right\},
\end{equation}
where, writing
\begin{equation}
    \label{eq:FunctionExpansion}
    w = \sum_{k=1}^{+\infty}w_k e_k \qquad w_k:=(w,\,e_k)_{L^2(\Omega)}\in L^2(0,T),\quad k\ge 1
\end{equation}
the norm is given by
\begin{equation}
    \label{eq:NormRepresentation}
    \|w\|_{L^2(\mathbb H^s)}^2 := \sum_{k=1}^{+\infty} \lambda_k^s\|w_k\|^2_{L^2(0,T)}.
\end{equation}
According to \cite[Section $2.5.1$, Appendix A]{Lischke.Pang.Gulian:20} and \cite[Chapter 1, Theorem 12.6]{Lions.Magenes:72}, denoting with $[\cdot,\,\cdot]_s$ the intermediate (i.e. real interpolation) space of index $s$, it holds
\begin{equation}
    \label{eq:InterpolationSpaces}
        L^2(\mathbb H^s(\Omega))  = L^2([H^1_0(\Omega), L^2(\Omega)]_s) = [L^2(H^1_0(\Omega)),\,L^2(L^2(\Omega))]_s.
\end{equation}
We denote the corresponding dual spaces as
\begin{equation*}
L^2(\mathbb H^{-s}(\Omega)) := L^2(\mathbb H^s(\Omega))^*.
\end{equation*}
Introducing the coefficients $g_k := \langle g, e_k\rangle \in L^2(0,T)$, $k \geq 1$, for $g \in L^2(\mathbb H ^{-s})$, a duality argument reveals
\begin{equation}
\label{eq:DualNormRepresentation}
    \|g\|_{L^2(\mathbb H^{-s})}^2 = \sum_{k=1}^{+\infty} \lambda_k^{-s}\|g_k\|^2_{L^2(0,T)}.
\end{equation}

\subsection{Eigenfunctions expansion}
\label{SS:EigenfunctionsExpansion}
The IBVP \eqref{eq:model-problem-weak} can be equivalently rephrased as an infinite system of second-order linear IVPs, by expanding the solution and the datum on the eigenfunctions of the Laplacian, cf. \cite[Chapter XV]{Dautray.Lions:92}. Remarkably, many properties of the IBVP can be inferred by just analyzing one such IVP, with the advantage of having an explicit solution formula, see \eqref{eq:model-problem-waek-reduced-solution} below. We refer, e.g., to \cite{Zank:20, Steinbach.Zank:22, Hoonhout.Loscher.UrzuaTorres:25} for previous uses of this observation.

Focusing on the $k$-th component in the eigenfunction expansion, $k \geq 1$, the spaces $\trial$, $\test$ and $\data$ are reduced to 
\begin{equation*}
\trial_k := \{ \widetilde u_k \in H^2(0,T) \mid \widetilde u_k(0) = 0 = \widetilde u_k'(0) \} 	
\end{equation*}
with
\begin{equation}
	\label{eq:norm-trial-exp}
\|\cdot\|^2_{\trial_k} := \lambda_k \|\cdot\|^2_{L^2(0,T)} + \frac{1}{c^4 \lambda_k} \|\partial_{tt} \cdot\|^2_{L^2(0,T)}
\end{equation}
as well as (see \eqref{eq:NormRepresentation} and \eqref{eq:DualNormRepresentation})
\begin{equation}
	\label{eq:spaces-expanded}
	\begin{aligned}
		&\test_k := L^2(0,T) && \text{with} && \|\cdot\|^2_{\test_k} := \lambda_k \|\cdot\|^2_{L^2(0,T)},\\
		&\data_k := L^2(0,T) && \text{with} && \|\cdot\|^2_{\data_k} := \frac{1}{\lambda_k} \|\cdot\|^2_{L^2(0,T)},
	\end{aligned}
\end{equation}
Analogously, the IBVP \eqref{eq:model-problem-weak} is reduced as follows: find $u_k \in \trial_k$ such that 
\begin{equation}
	\label{eq:model-problem-weak-reduced}
	\mathcal{L}_k u_k := \frac{1}{c^2}u_k'' +\lambda_k u_k = f_k \quad \text{in $\data_k$}.
\end{equation}
By Duhamel's principle, the solution of this IVP can be represented as
\begin{equation}
	\label{eq:model-problem-waek-reduced-solution}
	u_k(t) = \frac{c^2}{\sqrt{\mu_k}}\int_0^t f_k(s) \sin{(\sqrt{\mu_k}(t-s))}\,ds
\end{equation}
where
\begin{equation}
\label{eq:mu-k}
	\mu_k := c^2 \lambda_k
\end{equation}

with $\sqrt{\mu_k}$ playing the role of resonance frequency in  \eqref{eq:model-problem-weak-reduced}. Thus, the question of whether the IBVP \eqref{eq:model-problem-weak} establishes an isomorphism between $\trial$ and $\data$ is equivalent to asking whether the IVP \eqref{eq:model-problem-weak-reduced} establishes an isomorphism between $\trial_k$ and $\data_k$, with the caveat that the equivalence
\begin{equation}
\label{eq:isomorphism-k-mode}
  \| u_k \|_{\trial_k} \eqsim \|f_k\|_{\data_k}.
\end{equation}
holds with a constant independent of  $\lambda_k$.

\begin{remark}[Boundary conditions]
\label{rem:BoundaryConditions}
The eigenfunction expansion discussed in this section and our subsequent analysis do not directly extend to the case of the impedance boundary condition
\begin{equation*}
\frac{\theta}{c} \partial_t u + \partial_n u  = 0 \quad \text{in $\partial \Omega \times (0,T)$},
\end{equation*}
where $\theta > 0$. Note, however, that the extreme cases $\theta = 0$ and $\theta = +\infty$ correspond to Neumann and Dirichlet boundary conditions, to which our arguments apply.
\end{remark}

\section{Inspecting Zank's counter-example}
\label{S:Zank-counter-example}
This section investigates the mechanism behind the counter-example by Zank \cite[Theorem 4.2.23]{Zank:20}, elucidating the fundamental role played by resonance. The main result is Corollary~\ref{C:inf-sup-instability}, stating that the bilinear form induced by the operator $\mathcal{L}$ in \eqref{eq:model-problem-weak} does not fulfill  an inf-sup condition with respect to the norms $\|\cdot\|_\trial $ and $ \|\cdot\|_\test$ defined above. Therefore, the weak formulation proposed in Section~\ref{SS:functional-setting} for our model IBVP is not well-posed in the sense of the equivalence \eqref{eq:isomorphism}.

\subsection{Resonant waves}
\label{SS:resonant-waves} 

Let $k \geq 1$ and consider the IVP \eqref{eq:model-problem-weak-reduced} with source term
\begin{equation}
	\label{eq:source-term-wave}
	f_k(t) = f_{\omega}(t) := \cos(\omega t), \quad \omega > 0.
\end{equation}
Owing to Fourier analysis \cite{Dym.McKean:72}, any other source term in $\data_k$ can be decomposed into an infinite superposition of such elementary sources. The corresponding solution is given by
\begin{equation}
	\label{eq:solution-wave}
	u_k(t) = u_{k,\omega}(t) := \begin{cases}
		 \displaystyle \frac{-2c^2}{\mu_k-\omega^2}\sin\left(\frac{\omega + \sqrt{\mu_k}}{2}t\right)\sin\left(\frac{\omega-\sqrt{\mu_k}}{2}t\right) & \text{if } \omega \neq \sqrt{\mu_k} \\[10pt]
		 \displaystyle\frac{c^2 t}{2\sqrt{\mu_k}}\sin(\sqrt{\mu_k} t) & \text{if } \omega = \sqrt{\mu_k}.
	\end{cases}
\end{equation}
thanks to the representation formula \eqref{eq:model-problem-weak-reduced} and elementary calculations.

As long as $\omega \neq \sqrt{\mu_k}$, the solution exhibits bounded oscillation. Still, the oscillation grows in amplitude for $\omega \to \sqrt{\mu_k}$. In the limiting case $\omega = \sqrt{\mu_k}$, the source  frequency equals the resonance frequency, giving rise to a resonant solution , where the amplitude of the oscillation grows proportionally to time, see Figure~\ref{fig:resonance-illustration}.

\begin{figure}[ht!]
	\centering
	\includegraphics[width=1\linewidth]{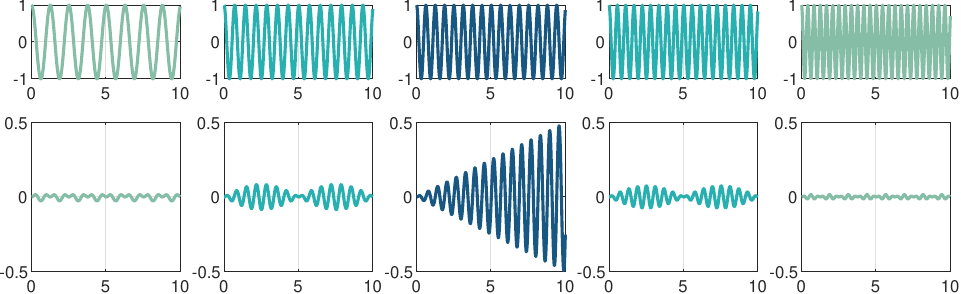}
	\caption{Source term $f_\omega$ (first row) and solution $u_{k,\omega}$ (second row) for the IVP \eqref{eq:model-problem-weak-reduced} in a progression from a non resonant case to the resonant one and back.}
	\label{fig:resonance-illustration}
\end{figure}

This class of source terms is not fully representative, since it includes only isolated frequencies but not arbitrary linear combinations of them, cf. Section~\ref{SS:FourierExpansion} below. Nevertheless, it is large enough to illustrate the mechanism behind the counter-example by Zank. The norm $\|\cdot\|_{\data_k}$ proposed in Section~\ref{SS:EigenfunctionsExpansion} does not distinguish if the source  frequency is close or not to the resonance frequency. Conversely, it must be expected that the trial norm $\|\cdot\|_{\trial_k}$ makes a substantial distinction between resonant and non-resonant waves, due to the growing amplitude of the oscillation in the former case. The next subsection complements this intuition by precise calculations. All results have been double-checked with the help of the Matlab Symbolic Math Toolbox. A copy of the source code is available at the public repository
\begin{center}
\label{url}
	\url{https://github.com/DavideTerazzi/resonant_waves_symbolic_tool.git}
\end{center}

\subsection{Inf-sup instability} 
\label{SS:inf-sup-instability}
Our first step consists in establishing a relation between the norm of the source term $f_\omega$ and the one of the corresponding solution $u_{k,\omega}$ of the IVP \eqref{eq:model-problem-weak-reduced}.

\begin{proposition}[Amplification constant]
\label{P:amplification-constant}
Let $f_\omega$ and $u_{k,\omega}$ be as in \eqref{eq:source-term-wave} and \eqref{eq:solution-wave}. The respective norms from Section~\ref{SS:EigenfunctionsExpansion} satisfy the relation
\begin{equation*}
\label{eq:amplification-constant}
	\|u_{k,\omega}\|^2_{\trial_k} = (1+C_{k,\omega})\, \|f_\omega\|^2_{\data_k}.
\end{equation*}
The constant $C_{k,\omega} \in \R$ is given by
\begin{subequations}
	\label{eq:amplification-constant-def}
\begin{equation}
\label{eq:amplification-constant-def-res}
\begin{aligned}
    C_{k,\omega} & :=  \frac{\mu_k}{1 + \sinc(2\sqrt{\mu_k} T)}\Bigg[\frac{T^2}{6} -\frac{T^2}{2}\sinc(2\sqrt{\mu_k} T) \\ & - \frac{1}{4\mu_k}\cos(2\sqrt{\mu_k} T) + \frac{1}{4\mu_k}\sinc(2\sqrt{\mu_k} T)\Bigg]
\end{aligned}
\end{equation}
for $\omega = \sqrt{\mu_k}$ and
\begin{equation}
\label{eq:amplification-constant-def-nonres}
\begin{aligned}
C_{k,\omega} & :=  \frac{2\mu_k}{1 + \sinc(2\omega T)} \frac{\mu_k + \omega^2}{(\mu_k - \omega^2)^2}  \Bigg[ 1 - \sinc((\sqrt{\mu_k} + \omega)T) \\ & - \sinc((\sqrt{\mu_k} - \omega)T) + \frac{\mu_k\sinc(2\sqrt{\mu_k} T) + \omega^2\sinc(2\omega T)}{\mu_k + \omega^2} \Bigg].
\end{aligned}
\end{equation}
\end{subequations}
otherwise.
\end{proposition}

\begin{proof}
We insert \eqref{eq:model-problem-weak-reduced} into \eqref{eq:norm-trial-exp}. Elementary manipulations reveal
\begin{equation*}
\|u_{k,\omega}\|^2_{\trial_k} 
=
\frac{1}{\lambda_k} \|f_\omega\|^2_{L^2(0,T)} + 2\lambda_k \|u_{k,\omega}\|^2_{L^2(0,T)} - 2 (f_\omega, u_{k,\omega})_{L^2(0,T)}. 
\end{equation*}
In view of \eqref{eq:source-term-wave}, the first term on the right-hand side is such that
\begin{equation*}
	\frac{1}{\lambda_k}\|f_\omega\|^2_{L^2(0,T)} = \frac{T}{2\lambda_k}\left(1+\sinc(2\omega T)\right)
\end{equation*}
For $\omega \neq \sqrt{\mu_k}$, the second term satisfies
\begin{align*}
    2\lambda_k\|u_{k,\omega}\|_{L^2(0,T)}^2 & =  \frac{2c^2\mu_kT}{(\mu_k-\omega^2)^2} \Bigg[ 1 + \frac{1}{2}\sinc(2\omega T)+ \frac{1}{2}\sinc(2\sqrt{\mu_k}T) \\ &  -\sinc((\sqrt{\mu_k}+\omega)T) -\sinc((\sqrt{\mu_k}-\omega)T)\Bigg].
\end{align*}
according to the first case in \eqref{eq:solution-wave}. Finally, for the third term, we have
\begin{align*}
    2(f_\omega,u_{k,\omega})_{L^2(0,T)} & = \frac{c^2(\mu_k-\omega^2)T}{(\mu_k-\omega^2)^2} \Bigg[ 1 + \sinc(2\omega T) \\ &  -\sinc((\sqrt{\mu_k}+\omega)T) -\sinc((\sqrt{\mu_k}-\omega)T)\Bigg].
\end{align*}
The combination of the above identities and the second line in \eqref{eq:spaces-expanded} yield the claimed expression of the constant $C_{k,\omega}$. For $\omega = \sqrt{\mu_k}$, we proceed by the same argument, using the second case in \eqref{eq:solution-wave}
\end{proof}

Our next step consists in clarifying how the constant $C_{k, \omega}$ from Proposition~\ref{P:amplification-constant} depends on the frequency $\omega$. 

\begin{proposition}[Frequency dependence]
\label{P:frequency-dependence}
Let $C_{k, \omega } \in \R$ be as in \eqref{eq:amplification-constant-def}. The mapping $(0,+\infty) \ni \omega \mapsto C_{k,\omega}$ is continuous and we have
\begin{equation*}
\label{eq:frequency-dependence}
\left |\lim_{\omega \searrow 0} C_{k,\omega}\right| \leq C 
\qquad \text{ and } \qquad 
\lim_{\omega \nearrow +\infty} C_{k,\omega} = 0
\end{equation*}
with the constant $C$ depending only on $T$ and $c$.
\end{proposition}
\begin{proof}
For the continuity, it suffices showing that the right-hand side of \eqref{eq:amplification-constant-def-nonres} converges to the one of \eqref{eq:amplification-constant-def-res} as $\omega \to \sqrt{\mu_k}$. By applying several high-order Taylor expansions around the point $\omega = \sqrt{\mu_k}$, we obtain
\begin{equation*}
	\lim_{\omega \to \sqrt{\mu_k}}\frac{\mu_k+\omega^2}{(\mu_k-\omega^2)^2}\Bigg[1-\sinc((\sqrt{\mu_k}-\omega)T)\Bigg] = \frac{T^2}{12}
\end{equation*}
and
\begin{equation*}
	\begin{aligned}
		\lim_{\omega \to \sqrt{\mu_k}} &  \frac{\mu_k+\omega^2}{(\mu_k-\omega^2)^2}\Bigg[-\sinc((\sqrt{\mu_k}+\omega)T)+ \frac{\mu_k\sinc(2\sqrt{\mu_k} T) + \omega^2\sinc(2\omega T)}{\mu_k + \omega^2}\Bigg] = \\ & = -\frac{T^2}{4}\sinc(2\sqrt{\mu_k} T) - \frac{1}{8\mu_k}\cos(2\sqrt{\mu_k} T) + \frac{1}{8\mu_k}\sinc(2\sqrt{\mu_k} T).
	\end{aligned}
\end{equation*}
Using these identities with \eqref{eq:amplification-constant-def-nonres} readily yields \eqref{eq:amplification-constant-def-res}. 

For the limit $\omega \searrow 0$, first-order Taylor expansions reveal
\begin{equation*}
	\lim\limits_{\omega \searrow 0} C_{k, \omega} = 1-2\sinc(\sqrt{\mu_k}T)+\sinc(2\sqrt{\mu_k}T).
\end{equation*}
We obtain the claimed estimate by invoking the boundedness of the function sinc.

Finally, the limit $\omega \nearrow +\infty$ is obtained by elementary arguments.
\end{proof}

We complement Proposition~\ref{P:amplification-constant} by plotting the function $\omega \mapsto 1+C_{k, \omega}$ for increasing values of $k$, see Figure~\ref{fig:amplification-constant}. We observe a bell-shaped profile, with the maximum point around the resonant frequency $\omega = \sqrt{\mu_k}$. The maximum  value grows asymptotically as $(T\sqrt{\mu_k})^2/6$ for $k\to +\infty$, in agreement with \eqref{eq:amplification-constant-def-res}. The limits for $\omega \searrow 0$ and $\omega \nearrow +\infty$ appear to remain bounded irrespective of $k$ as expected.

\begin{figure}[ht!]
	\centering
	\includegraphics[width=0.6\linewidth]{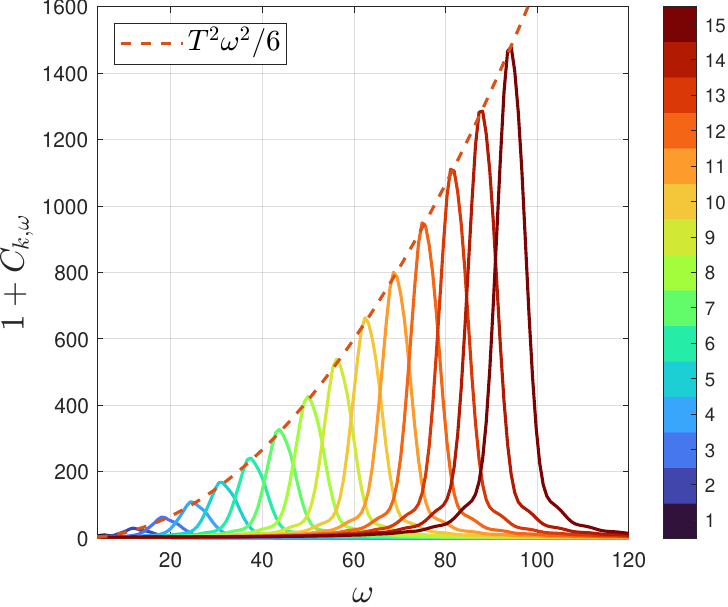}
	\caption{Plot of the amplification constant in Proposition \ref{P:amplification-constant} for $\sqrt{\mu_k} = 2\pi k$ with $k = 1,\dots,15$. The final time is fixed to $T = 1$.}
	\label{fig:amplification-constant}
\end{figure}

Owing to the above discussion, we are in position to state our main result concerning the well-posedness of the IBVP \eqref{eq:model-problem-weak} in the setting of Section~\ref{SS:functional-setting}. This is formulated in terms of the bilinear form $\mathcal{B}: \trial \times \test \to \R$ induced by the operator $\mathcal{L}$, namely
\begin{equation}
\label{eq:bilinear-form}
\mathcal{B}(\widetilde u, v) := \int_0^T \Big ( \left\langle \partial_{tt} \widetilde u, v\right\rangle + (\nabla \widetilde u, \nabla v)_{L^2(\Omega)} \Big) \dt
\end{equation} 
with $\left\langle \cdot, \cdot\right\rangle$ denoting the dual pairing of $H^{-1}(\Omega)$ and $H^1_0(\Omega)$.

\begin{corollary}[Inf-sup instability]
\label{C:inf-sup-instability}
The bilinear form in \eqref{eq:bilinear-form} above is such that
\begin{equation}
\label{eq:inf-sup-instability}
\inf_{\widetilde u \in \trial}\,\,\sup _{v \in \test} \,\frac{\mathcal{B}(\widetilde u,v)}{\|\widetilde u\|_\trial\|v\|_\test} = 0.
\end{equation}
Consequently, the IBVP \eqref{eq:model-problem-weak} does not satisfy the equivalence estimate \eqref{eq:isomorphism}.
\end{corollary}

\begin{proof}
Recall the eigenfunctions $(e_k)_{k=1}^{+\infty}$ of the Laplacian. For $k\geq 1$, consider the functions 
\begin{equation}
\label{eq:inf-sup-instability-proof}
\widetilde{u} = u_{k,\omega} e_k \in \trial
\qquad \textit{and} \qquad
\widetilde f = f_\omega e_k \in \data
\end{equation}
where $u_{k,\omega}$ and $f_\omega$ are as in \eqref{eq:solution-wave} and \eqref{eq:source-term-wave}, respectively, with $\omega > 0$. Since $u_{k, \omega}$ solves the IVP \eqref{eq:model-problem-weak-reduced} with source $f_\omega$, it follows that $\widetilde u$ solves the IBVP \eqref{eq:model-problem-weak} with source $\widetilde f$, hence
\begin{equation*}
	\mathcal{B} (\widetilde u, v) = \int_0^T \langle \widetilde f, v\rangle \dt \qquad\forall v \in \test.
\end{equation*}
This identity, definitions \eqref{eq:WaveTestDataNorm}, \eqref{eq:norm-trial-exp}, \eqref{eq:spaces-expanded} and \eqref{eq:inf-sup-instability-proof} and Proposition~\ref{P:amplification-constant} reveal
\begin{equation*}
\begin{split}
	\inf_{\widetilde u \in \trial}\,\,\sup _{v \in \test} \,\frac{\mathcal{B}(\widetilde u,v)}{\|\widetilde u\|_\trial\|v\|_\test} 
	&\leq
\sup _{v \in \test} \,\frac{\int_0^T \langle \widetilde f, v\rangle \dt}{\|\widetilde u\|_\trial\|v\|_\test} 
= 
\frac{\|f_\omega\|_{\data_k}}{\|u_{k,\omega}\|_{\trial_k}}
=
\frac{1}{1 + C_{k, \omega}}.
\end{split}
\end{equation*}
Taking $\omega = \sqrt{\mu_k}$, Proposition~\ref{P:frequency-dependence} entails $C_{k, \omega} \to +\infty$ as $k \to +\infty$. Inserting this observation in the right-hand side of the upper bound above implies \eqref{eq:inf-sup-instability}, because the left-hand side is nonnegative. 

The identity \eqref{eq:inf-sup-instability} readily implies that the IBVP \eqref{eq:model-problem-weak} does not fulfill \eqref{eq:isomorphism}  due to the Banach-Ne\v{c}as-Babu\v{s}ka theorem~\cite[Theorem~25.9]{Ern.Guermond:21-II}.
\end{proof}

It is worth noticing that Zank \cite{Zank:20} used the argument in the proof of Corollary~\ref{C:inf-sup-instability} to infer inf-sup instability for a different weak formulation of the acoustic wave equation. Indeed, it appears to us that Zank's counter-example unveils a general distinguishing property of wave equations. We elaborate on this observation in the next remarks.

\begin{remark}[Space-time coupling] 
\label{rem:SpaceTimeCoupling}
Observe that resonance can be understood as a deep interplay between the temporal and spatial frequencies, cf. \eqref{eq:solution-wave}. It manifests when a specific temporal frequency of the forcing term aligns with one of the resonance frequency $\sqrt{\mu_k}$ of the underlying operator. This synchronization triggers an amplification mechanism that is inherently space-time coupled, suggesting that the resonant growth of the solution cannot be completely factored into independent spatial and temporal components.
\end{remark}

\begin{remark}[Resonance-aware norms]
\label{rem:resonance-aware-norms}
In the functional setting considered so far, the trial norm is of Bochner type. As suggested by Figure \ref{fig:resonance-illustration}, this norm is able to distinguish between resonant and non-resonant solutions. Consequently, for the equivalence \eqref{eq:isomorphism} to hold, the data norm must also possess this capability, i.e. it must assign a larger weight to the source terms whose  frequency is close to the resonance frequency, cf. Figure~\ref{fig:amplification-constant}. However, as shown by the representation in \eqref{eq:NormRepresentation} the proposed data norm fails in this regard. The reason for this failure lies in the tensor-product structure of Bochner-type norms, which treats space and time as separate variables. This separation is in direct conflict with the inherently coupled nature of resonance described in Remark \ref{rem:SpaceTimeCoupling}. To correctly capture the frequency matching and restore the isomorphism, it is therefore necessary to introduce a  resonance-aware norm that explicitly accounts for this coupling. Finally, although our analysis is focused on the specific setting of Section \ref{SS:functional-setting}, we expect that this is a general phenomenon, that is the mechanism responsible for resonance is not fully captured whenever both trial and data spaces are equipped with standard, Bochner-type norms.  
\end{remark} 

\begin{remark}[Comparison with classical results]
\label{rem:comparison-with-classical-results}
In classical results \cite{Dautray.Lions:92, Evans:10, Ladyzhenskaya:85, Lions.Magenes:72}, existence and uniqueness of the solution of \eqref{eq:model-problem-weak} are guaranteed for source terms $f \in L^2(L^2(\Omega))$, together with the estimate
\begin{equation}
\label{eq:comparison-with-classical-results}
\|u\|_\trial \leq C \|f\|_{L^2(L^2)}
\end{equation}
where the constant $C$ depends on $T$ and $c$. As we have mentioned already in Remark~\ref{rem:data-space}, this estimate is not sharp, in the sense that it cannot be reversed. It is worth reviewing this observation in the light of the above discussion. The restriction of the $L^2(L^2)$-norm along the $k$-th component in the eigenfunction expansion \eqref{eq:NormRepresentation} equals $\|\cdot\|_{L^2(0,T)}$, i.e. it differs from the restriction of the $L^2(H^{-1})$-norm only by a factor $\lambda_k$, cf. \eqref{eq:spaces-expanded}. Such a factor corresponds to the worst case for the constant $C_{k,\omega}$ in Proposition~\ref{P:amplification-constant}, indicating that all frequencies $\omega$ in \eqref{eq:source-term-wave} are treated as resonance frequencies by the the $L^2(L^2)$-norm. Thus, the upper bound in \eqref{eq:comparison-with-classical-results} is sharp in some cases and provides a substantial overestimation in other ones.
\end{remark}

\begin{figure}[ht!]
	\centering
	\includegraphics[width=0.6\linewidth]{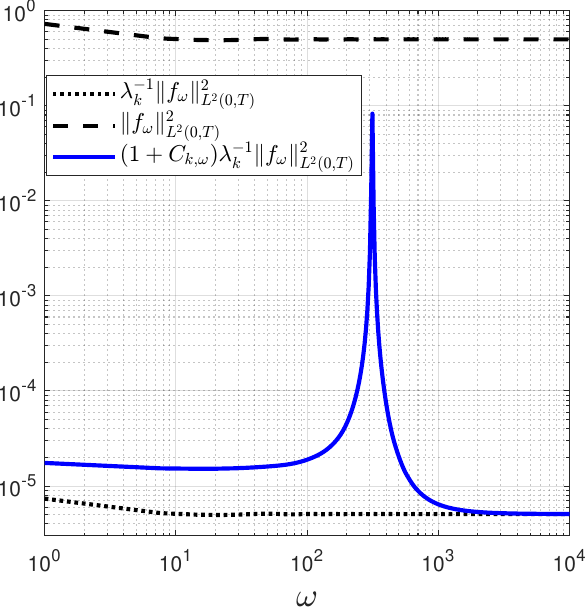}
	\caption{Comparing the data norm of $f_\omega$ from Proposition \ref{P:amplification-constant} with the $L^2(L^2)-$ and $L^2(H^{-1})-$norm, for $k=50$ and $\sqrt{\mu_k}=2\pi k$.}
	\label{fig:bounds-illustration}
\end{figure}

Figure~\ref{fig:bounds-illustration} summarizes the results and the discussion from this section in one plot. For the IVP \eqref{eq:model-problem-weak-reduced} with $k\geq 1$, the (restrictions of both the) $L^2(H^{-1})$- and the $L^2(L^2)$-norm are almost constant for all source terms in \eqref{eq:source-term-wave} and differ only by the multiplicative factor $\lambda_k$. None of these norms is suitable to establish the equivalence \eqref{eq:isomorphism}, because the trial norm of the corresponding solutions is by far not constant. In particular, the latter one approximately behaves as the $L^2(H^{-1})$-norm of the source term away from the resonance frequency $\sqrt{\mu_k}$ and it almost coincides with the $L^2(L^2)$-norm around $\sqrt{\mu_k}$. 

Moreover, the same behaviour arises for all intermediate Bochner-type $L^2(\mathbb H^s)$-norms, $s \in (0,1)$: along the $k$-th Fourier component such norms differ only by the multiplicative factor $\lambda_k^{\,s}$ (cf. \eqref{eq:NormRepresentation}), hence they are
essentially insensitive to the  frequency $\omega$. In particular, they
cannot reproduce the bell-shaped dependence of $\|u_{k,\omega}\|_{\trial_k}$ on
$\omega$, confirming that the standard Bochner scale for the data space is not flexible enough to
capture resonant effects.

\section{Looking for resonance-aware data norms}
\label{S:resonance-aware-norms} 
In this section we search for an alternative to the data norm $\|\cdot\|_{\data_k}$ from \eqref{eq:spaces-expanded} for the IVP \eqref{eq:model-problem-weak-reduced} with fixed $ k\geq 1$. We aim at establishing a counterpart of the equivalence \eqref{eq:isomorphism-k-mode} with constants independent of $\lambda_k$.

\begin{remark}[Restriction to the IVP]
	\label{rem:restriction-IVP}
In Section~\ref{S:Zank-counter-example}, restricting our attention from the IBVP \eqref{eq:model-problem} to the IVP \eqref{eq:model-problem-weak-reduced} is only instrumental to a more concise presentation of Zank's counter-example and it does not affect the generality of our main result, cf. Corollary~\ref{C:inf-sup-instability}. In contrast, the restriction to the IVP \eqref{eq:model-problem-weak-reduced} is a really simplifying assumption here. In fact, having removed the dependence on the space variable, we are not faced with delicate questions such as the characterization of the data space and the existence of a solution. Taking into account also that we deal with specific boundary conditions, see Remark~\ref{rem:BoundaryConditions}, our next results should be accounted only as a first step towards a suitable functional framework for the analysis of wave equations.
\end{remark}

Let us denote by $\|\cdot\|_{k,*}$ the desired norm on the data space $\data_k$. According to Corollary~\ref{C:inf-sup-instability} and Remarks~\ref{rem:resonance-aware-norms} and \ref{rem:comparison-with-classical-results}, this norm must be resonance-aware and it must satisfy the two-sided bound
\[
\frac{C_1}{\lambda_k} \|\cdot\|_{L^2(0,T)} \leq \|\cdot\|_{k, *} \leq C_2\|\cdot\|_{L^2(0,T)} 
\] 
for constants $0 < C_1 \leq C_2$ independent of $\lambda_k$.

Since the operator $\mathcal{L}_k$ in \eqref{eq:model-problem-weak-reduced} is invertible, we readily obtain the identity
\begin{equation}
\label{eq:optimal-data-norm}
\|u_k \|_{\trial_k} = \|\mathcal{L}_k^{-1} f_k\|_{\trial_k} 
\end{equation}
showing that the right-hand side is the optimal choice for the norm $\|\cdot\|_{k, *}$. Indeed, it makes $\mathcal{L}_k$ not only an isomorphism but even an isometry. This idea has been used in a number of recent references \cite{Bignardi.Moiola:25, Fuhrer.Gonzalez.Karkulik:25, Steinbach.Zank:22, Zank:20, Henning.Palitta.Simoncini.Urban:22}. Still, as we mentioned also in the introduction, we are interested in devising more explicit and operative equivalent definitions, that describe the nature of $\|\cdot\|_{k,*}$ without involving the operator $\mathcal{L}_k$. 

We present  two possible approaches for the construction $\|\cdot\|_{k,*}$  in the next subsections. In the first approach, we expand the source term $f$ into classical Fourier series,
we then define the norm in terms of the coefficients of this expansion, 
weighting more  the components with frequencies near the resonance frequency $\sqrt{\mu_k}$. The second approach, however, is 
based on a suitable energy balance law and an appropriate transform of 
the data, which depends on $\sqrt{\mu_k}$.

\subsection{First approach: Fourier expansion}
\label{SS:FourierExpansion}
As we observed in Section~\ref{SS:resonant-waves}, the source terms in \eqref{eq:source-term-wave} are not fully representative but, according to Fourier analysis, any admissible source term $f_k \in \data_k$ can be decomposed into an infinite combination of (some of) those ones. For instance, following \cite{Zank:20}, we have
\begin{equation}
	\label{eq:fourier-expansion}
	f_k(t) =\sum_{j=1}^{+\infty} f_{k,j}\cos(\omega_j t) 
\end{equation}
where the frequencies $(\omega_j)_{j=1}^{+\infty} \subseteq (0, +\infty)$ and the coefficients $(f_{k,j})_{j=1}^{+\infty} \subseteq \R$ are given by
\begin{equation}
\label{eq:fourier-expansion-coefficients}
	f_{k, j} := \int_0^T f_k(t)\cos(\omega_j t)\dt 
	\qquad \text{and} \qquad 
	\omega_j := \frac{1}{T}\left(\frac{\pi}{2}+j\pi\right)
\end{equation}
for $j \geq 1$. Of course, using a different Fourier basis would be possible, cf. Remark~\ref{rem:fourier-basis} below.
 
Combining the solution formula \eqref{eq:solution-wave} with the expansion above, we obtain an explicit expression for the desired data norm $\|\cdot\|_{k,*}$.
  
\begin{proposition}[Data norm by Fourier expansion]
	\label{P:data-norm-fourier-expansion}
	Let $u_{k} \in \trial_k$ and $f_k \in \data_k$ be the solution and the source term of the IVP \eqref{eq:model-problem-weak-reduced}. We have
    \begin{equation}
		\label{eq:data-norm-fourier-expansion}
		\|u_k\|^2_{\trial_k} = \sum_{j,\ell=1}^{+\infty} \mathcal{W}_{k}(\omega_j,\,\omega_\ell)\,f_{k,j} f_{k,\ell} 
	\end{equation}
    where $(0,\,+\infty)\times (0,\,+\infty)\ni (\omega,\,\widetilde \omega)\mapsto \mathcal{W}_{k}(\omega,\,\widetilde\omega)$ is the continuous map given by
    \begin{equation}
    \label{eq:data-norm-fourier-expansion-def}
        \begin{aligned}
            \mathcal{W}_{k}(\omega,\,\widetilde\omega) & := \frac{T}{2\lambda_k}\left(\sinc((\omega+\widetilde \omega)T)+\sinc((\omega-\widetilde \omega)T)\right) \\ & + \frac{c^2T}{(\mu_k - \omega^2)(\mu_k-\widetilde\omega^2)} \Bigg[ \mu_k(1+\sinc(2\sqrt{\mu_k}T)) \\ & + \frac{\omega^2+\widetilde\omega^2}{2}(\sinc((\omega-\widetilde\omega)T)+\sinc((\omega+\widetilde\omega)T))
            \\ & - \frac{\mu_k+\omega^2}{2}\left(\sinc((\sqrt{\mu_k} - \omega)T) + \sinc((\sqrt{\mu_k} + \omega)T)\right) \\ & -\frac{\mu_k+\widetilde\omega^2}{2}\left(\sinc((\sqrt{\mu_k} - \widetilde\omega)T)+\sinc((\sqrt{\mu_k} + \widetilde\omega)T)\right)\Bigg].
        \end{aligned}
    \end{equation}
    In particular, for the diagonal elements it holds that
    \begin{equation}
        \label{eq:data-norm-fourier-expansion-def-diag}
        \mathcal{W}_{k}(\omega_j,\,\omega_j) = (1+C_{k,\omega_j})\|f_{\omega_j}\|^2_{\data_k}
    \end{equation}
    for $j \geq 1$, where $f_{\omega_j}$ and $C_{k,\omega_j}$ are as in \eqref{eq:source-term-wave} and in Proposition \ref{P:amplification-constant}, respectively.
\end{proposition}  
  
\begin{proof} Arguing as in the proof of Proposition~\ref{P:amplification-constant}, we infer the identity
\begin{equation}
	\label{eq:data-norm-fourier-expansion-proof}
\|u_{k}\|^2_{\trial_k} 
=
\frac{1}{\lambda_k} \|f_k\|^2_{L^2(0,T)} + 2\lambda_k \|u_{k}\|^2_{L^2(0,T)} - 2 (f_k, u_{k})_{L^2(0,T)}. 
\end{equation}
Then, we compute each term on the right-hand side with the help of \eqref{eq:solution-wave} and \eqref{eq:fourier-expansion}. The continuity of $(\omega,\widetilde\omega)\mapsto\mathcal W_k(\omega,\widetilde\omega)$ can be established by the same reasoning as in the proof of Proposition \ref{P:frequency-dependence}.
 We omit further details and refer to the source code at our public repository (see the url address on page \pageref{url}) for a verification by symbolic computations.
\end{proof} 

The Fourier expansion \eqref{eq:fourier-expansion} induces an isometry between $L^2(0,T)$ and the space $\ell^2$ of square-summable sequences. The right-hand side of \eqref{eq:data-norm-fourier-expansion} provides an expression for the data norm $\|\cdot\|_{k,*}$, representing it through a positive-definite quadratic form $\mathcal{Q}$ on $\ell^2$. Figure~\ref{fig:quadratic-form-coefficients} provides an intuition about the size of the first coefficients of $\mathcal{Q}$. Note, in particular, that the diagonal of the diagram exhibits the same bell-shaped profile from Figures~\ref{fig:amplification-constant} and \ref{fig:bounds-illustration}, cf. \eqref{eq:data-norm-fourier-expansion-def-diag}.

\begin{figure}[ht!]
	\centering
	\includegraphics[width=0.6\linewidth]{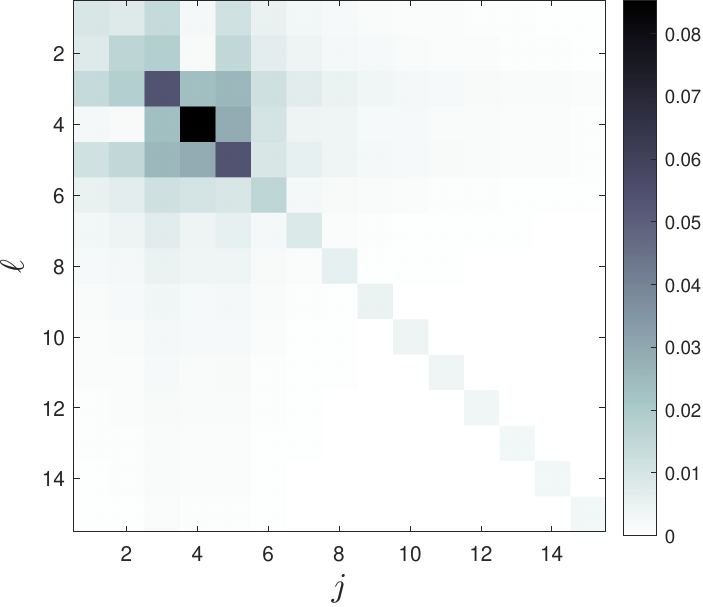}
	\caption{Absolute value of the first coefficients of $\mathcal{W}_{k}(\omega_j,\,\omega_\ell)$ from \eqref{eq:data-norm-fourier-expansion}, with  $\sqrt{\mu_k} =200$, and $T=1$, $c=1$.}
	\label{fig:quadratic-form-coefficients}
\end{figure}

In order to simplify the expression of the data norm, it looks reasonable asking if the form $\mathcal{Q}$ is equivalent to a diagonal quadratic form.  For this purpose, the most immediate candidate is the diagonal part of $\mathcal{Q}$, namely 
\begin{equation*}
\sum_{j = 1}^{+\infty} 
\mathcal{W}_{k}(\omega_j,\,\omega_j) f_{k,j}^2.
\end{equation*}
Unfortunately, we verified by symbolic computations that such an equivalence does not hold true. In particular, the form $\mathcal{Q}$ is not diagonally dominant.

\begin{remark}[Fourier basis]
\label{rem:fourier-basis}	
The expression of the coefficients $\mathcal{W}_{k}(\omega_j,\,\omega_\ell)$ in \eqref{eq:data-norm-fourier-expansion-def} and the discussion following Proposition~\ref{P:data-norm-fourier-expansion} critically depend on the specific Fourier basis at hand. We considered some other options, but none of the resulting quadratic forms appears to be equivalent to its diagonal part. Of course, a suitable basis for this purpose could be found by diagonalization, but it is our impression that the necessary change of basis is far from straight-forward.
\end{remark}

\subsection{Second  approach: energy balance law}  
\label{SS:energy-balance-law}
The starting point of our second approach is a point-wise in time energy balance law for the IVP \eqref{eq:model-problem-weak-reduced}. This law relates the total energy, obtained as the sum of potential and kinetic energy, to suitable transformations  of the source term. 

\begin{lemma}[Energy balance]
\label{L:energy-balance}
Let $u_{k} \in \trial_k$ and $f_k \in \data_k$ be the solution and the source term of the IVP \eqref{eq:model-problem-weak-reduced}. We have
\begin{equation}
\label{eq:energy-balance}
\lambda_k u_k(t)^2 + \frac{1}{c^2}u_k'(t)^2  = c^2\Big (\fcos(t)^2 + \fsin(t)^2 \Big ), \quad \text{a.e. in $(0,T)$}
\end{equation}
where $\fcos,\,\fsin \in H^1(0,T)$ are defined as
\begin{equation}
\label{eq:fcos-fsin}
\fcos(t)  := \int_0^t f_k(s)\cos{(\sqrt{\mu_k} s)} \mathrm{d}s  
\quad \text{and} \quad 
\fsin(t)  := \int_0^t f(s)\sin{(\sqrt{\mu_k} s)}\mathrm{d}s.
\end{equation}
\end{lemma}

\begin{proof}
Combining the solution formula \eqref{eq:solution-wave} with elementary trigonometric manipulations, we infer the vector identity
\begin{equation*}
	\frac{1}{c^2}
	\begin{pmatrix}
		\sqrt{\mu_k}u_k(t)\\[3pt]
		u_k'(t)
	\end{pmatrix}
	=
	\underbrace{\begin{pmatrix}
		\sin(\sqrt{\mu_k}t) & -\cos(\sqrt{\mu_k}t)\\[3pt]
		\cos(\sqrt{\mu_k}t) &  \sin(\sqrt{\mu_k}t)
	\end{pmatrix}}_{=: R_k(t)}
	\begin{pmatrix}
		\fcos(t)\\[3pt]
		\fsin(t)
	\end{pmatrix}, 
	\quad \text{a.e. in $(0,T)$}.
\end{equation*}
Since $R_k(t)$ is an orthogonal matrix, we obtain the claimed energy balance law by taking the squared Euclidean norm on both sides and recalling \eqref{eq:mu-k}.
\end{proof}

With the above balance law at hand, we are in position to provide an alternative to Proposition~\ref{P:data-norm-fourier-expansion} for the representation of the data norm $\|\cdot\|_{k,*}$.

\begin{proposition}[Data norm by energy balance]
	\label{P:data-norm-energy-balance}
	Let $u_{k} \in \trial_k$ and $f_k \in \data_k$ be the solution and the source term of the IVP \eqref{eq:model-problem-weak-reduced}. We have
	\begin{equation}
		\label{eq:data-norm-energy-balance}
		\|u_k\|^2_{\trial_k} + \frac{1}{c^2} \|u_k'\|^2_{L^2(0,T)} 
		\eqsim 
		\frac{1}{\lambda_k}\|f_k\|^2_{L^2(0,T)} +
		c^2\Big (\|\fcos\|^2_{L^2(0,T)} + \|\fsin\|^2_{L^2(0,T)} \Big )
	\end{equation}
	with hidden constants independent of $\lambda_k$, $c$ and $T$. 
\end{proposition}  

\begin{proof}
Integrating the point-wise energy balance \eqref{eq:energy-balance} over $(0,T)$, we obtain	
\[
\frac{\mu_k}{c^2}\|u_k\|^2_{L^2(0,T)} + \frac{1}{c^2}\|u_k'\|^2_{L^2(0,T)} 
=
c^2\Big (\|\fcos\|^2_{L^2(0,T)} + \|\fsin\|^2_{L^2(0,T)} \Big ).
\]
The combination of this identity with \eqref{eq:mu-k}, \eqref{eq:data-norm-fourier-expansion-proof} and simple Young's inequalities provides the claimed equivalence.
\end{proof}

According to Remark~\ref{rem:embedding-theorem} and \eqref{eq:NormRepresentation}, the norm on the left-hand side of \eqref{eq:data-norm-energy-balance} is equivalent to $\|\cdot\|_{\trial_k}$. The data norm on the right-hand side is resonance-aware in the sense of Remark~\ref{rem:resonance-aware-norms}, because the functions $\fcos$ and $\fsin$ incorporate the resonance frequency $\sqrt{\mu_k}$. A disadvantage of this representation of the data norm is that we could not find a simple expression for the for the dual norm, i.e. for the test norm to be used in place of $\|\cdot\|_{\test_k}$ from \eqref{eq:spaces-expanded}. An advantage is that we can easily incorporate nonzero initial values as follows.

\begin{remark}[Nonzero initial values]
\label{rem:nonzero-initial-values}
Assume that we have nonzero initial values 
\begin{equation*}
	u_k(0) = g_k \qquad \text{and} \qquad u_k'(0) = h_k
\end{equation*}
in the IVP \eqref{eq:model-problem-weak-reduced}. Then, the solution verifies the relation
\begin{equation*}
	\frac{1}{c^2}\begin{pmatrix}
		\sqrt{\mu_k}u_k(t) \\[3pt] 
		u_k'(t)
	\end{pmatrix} = \begin{pmatrix}
		\sin{(\sqrt{\mu_k}t)} & -\cos{(\sqrt{\mu_k}t)} \\[3pt] 
		\cos{(\sqrt{\mu_k}t)} &\sin{(\sqrt{\mu_k}t)}
	\end{pmatrix}\begin{pmatrix}
		\fcos(t) +\frac{1}{c^2}h_k \\[3pt] 
		\fsin(t)-\frac{\sqrt{\mu_k }}{c^2}g_k
	\end{pmatrix}, \quad \text{a.e. in $(0,T)$}
\end{equation*}
giving rise to the point-wise energy balance
\begin{equation*}
	\lambda_k u_k(t)^2 + \frac{1}{c^2}u_k'(t)^2  = c^2\Big (\Big(\fcos(t)+\frac{1}{c^2}h_k\Big)^2 + \Big( \fsin(t)-\frac{\sqrt{\mu_k}}{c^2} g_k \Big)^2 \Big ), \quad \text{a.e. in $(0,T)$}
\end{equation*}
and to the norm equivalence
\begin{equation*}
	\begin{split}
		&\|u_k\|^2_{\trial_k} + \frac{1}{c^2} \|u_k'\|^2_{L^2(0,T)} 
	\eqsim \\
	& \qquad  \frac{1}{\lambda_k}\|f_k\|^2_{L^2(0,T)} +
	c^2\Big (\left\|\fcos(t)+\frac{1}{c^2}h_k\right\|^2_{L^2(0,T)} + \left \| \fsin(t)-\frac{\sqrt{\mu_k}}{c^2} g_k\right \|^2_{L^2(0,T)} \Big ).
	\end{split}
\end{equation*}
Note that, in case $f=0$, we recover the well-known energy conservation law for the acoustic wave equation, cf. \cite[Chapter XV, Section $4$, Remark 5]{Dautray.Lions:92}.
\end{remark}

\subsection{Damped wave equation}
\label{SS:damping}
We now investigate how the previous analysis extends to the case in which damping effects are taken into account.
For this purpose, we consider the IBVP 
\begin{equation*}
	\begin{aligned}
		\begin{aligned}
			\partial_{tt} u +\rho\,\partial_t u - \Delta u & = f && \text{in } \Omega\times (0,T) \\
			u & = 0 && \text{in } \partial\Omega \times (0,T) \\
			u(\cdot, 0) & = 0 && \text{in } \Omega \\
			u_t(\cdot, 0) & = 0 && \text{in } \Omega
		\end{aligned}
	\end{aligned}
\end{equation*}
with $\rho>0$. 

Using the functional setting from Section~\ref{SS:functional-setting}, the eigenfunction expansion from Section~\ref{SS:EigenfunctionsExpansion} leads to the following IVP 
\begin{equation*}
	u_k'' + \rho\,u_k' +\lambda_k u_k = f_k
\end{equation*}
for $k \geq 1$. The corresponding solution is given by
\begin{align*}
	\begin{aligned}
		\displaystyle u_k(t) = \frac{2}{\sqrt{\rho^2-4\lambda_k}}\int_0^t \sinh\left(\frac{\sqrt{\rho^2-4\lambda_k}}{2}(t-s)\right) e^{\frac{\rho}{2}(s-t)}f_k(s)\mathrm{d}s && \text{if } 4\lambda_k< \rho^2, \\
		\displaystyle u_k(t) = \frac{2}{\sqrt{4\lambda_k-\rho^2}}\int_0^t \sin\left(\frac{\sqrt{4\lambda_k-\rho^2}}{2}(t-s)\right) e^{\frac{\rho}{2}(s-t)}f_k(s)\mathrm{d}s &&  \text{if } 4\lambda_k> \rho^2. 
	\end{aligned}
\end{align*}
The specific expression of the solution when $4\lambda_k = \rho^2$ is irrelevant for our discussion.

For $4\lambda_k < \rho^2$ there is no resonant solution and we could proceed in analogy with the heat equation, cf. Section~\ref{SS:Heat equation}. Still, for $4\lambda_k > \rho^2$, the source term
\[
f_k(t) = e^{\frac{\rho}{2}(T-t)}\sin\left(\frac{\sqrt{4\lambda_k-\rho^2}}{2}t\right)
\]
gives rise to a resonant solution. In view of \eqref{eq:eigenvalues}, this case eventually comes into place, irrespective of $\rho$, indicating that a counterpart of Corollary~\ref{C:inf-sup-instability} holds true.  

Upon setting
\[
\eta_k := \frac{4\lambda_k - \rho^2}{4} > 0 
\qquad \text{and} \qquad
\widetilde{f}_k(t) := f_k(t)e^{\frac{\rho}{2} t}, 
\]
a candidate resonance-aware data norm can be inferred by arguing, e.g., as in Section~\ref{SS:energy-balance-law}. Indeed, the solution verifies the relation
\[
\begin{pmatrix}
	\sqrt{\eta_k} u_k(t) \\[3pt] 
	u_k'(t)
\end{pmatrix} 
= 
\begin{pmatrix}
	\sin(\sqrt{\eta_k} t) & -\cos(\sqrt{\eta_k} t) \\[3pt] 
	\cos(\sqrt{\eta_k} t) & \sin(\sqrt{\eta_k} t)
\end{pmatrix}
\begin{pmatrix}
	e^{-\frac{\rho}{2} t}\ftcos(t) \\[3pt]
	e^{-\frac{\rho}{2} t} \ftsin(t)
\end{pmatrix}, \quad \text{a.e. in $(0,T)$}
\]
with the notation from \eqref{eq:fcos-fsin}. Thus, we obtain the point-wise energy balance 
\[
\eta_k u_k^2(t) + u_k'(t)^2 = e^{-\rho t}\left( \ftcos(t)^2+ \ftsin(t)^2\right), \quad \text{a.e. in $(0,T)$}
\]
and a corresponding counterpart of Proposition~\ref{P:data-norm-energy-balance}.

\section{Conclusion}
Our analysis points out a distinguishing property of the linear wave equation: the existence of resonant solutions prevents from establishing an isomorphism when both the trial and the data space in a weak formulation are equipped with classical Bochner-type norms. Indeed, resonant solutions hinge on nontrivial compensations between the different components of the differential operator, cf. Remark~\ref{rem:OperatorComponents}.

From a complementary viewpoint, resonance can be interpreted as a coupling phenomenon between temporal and spatial frequencies: it occurs when a  frequency of the forcing term matches with one of the spatial eigenfrequencies of the underlying operator. This matching triggers an amplification mechanism that is inherently space–time coupled. A Bochner-type norm treats the two variables separately, due to its tensor-product structure. Therefore, it cannot detect such frequency matching. As a consequence, using Bochner-type norms for both the trial and data spaces does not fully capture the mechanism responsible for resonance, highlighting the need for a norm that explicitly accounts for this coupling, cf. Remark \ref{rem:SpaceTimeCoupling} - \ref{rem:resonance-aware-norms}.

To address this, in a simplified setting we devised two approaches to construct resonance-aware norms. While the resulting norms may not be immediately convenient for the analysis of discretization methods, they highlight a fundamental aspect of the problem: resonances are intrinsic to the equation and must be explicitly taken into account in any rigorous analytical framework. This insight provides a clearer understanding of the structure of the wave equations and sets the stage for future work on a functional setting and numerical methods that take into account these resonant effects.

\subsection*{Acknowledgments} {The authors are members of the GNCS-INdAM. }

\appendix

\section{Extension to other IBVPs}
\label{A:extension-other-IBVPs}
In this appendix, we briefly discuss how the previous analysis can be adapted to other initial-boundary value problems.

\subsection{Heat equation}
\label{SS:Heat equation}
First of all, it is worth comparing with the following IBVP
\begin{equation*}
	\label{eq:heat}
	\begin{aligned}
		\begin{aligned}
			\partial_t u - \Delta u & = f && \text{in } \Omega \times (0,T), \\
			u & = 0 && \text{in } \partial\Omega \times (0,T), \\
			u(\cdot,0) & = 0 && \text{in } \Omega.
		\end{aligned}
	\end{aligned}
\end{equation*}
for the heat equation. Indeed, upon setting
\begin{equation*}
\trial := \{\widetilde{u} \in L^2(H^1_0(\Omega)) \cap H^1(H^{-1}(\Omega)) \mid \widetilde u(0) = 0\}
\qquad \text{and} \qquad
\test := L^2(H^1_0(\Omega))
\end{equation*}
well-posedness and \eqref{eq:isomorphism} are established for the following weak formulation:
Find $u \in \trial$ such that
\begin{equation}
	\label{eq:heat-weak}
	\mathcal{L} u := \partial_{t} u - \Delta u = f \quad \text{in $\data$}
\end{equation}
see, e.g., \cite[Section 65]{Ern.Guermond:21-III}

In contrast the acoustic wave equation, it is interesting that  \eqref{eq:isomorphism} can be obtained here using standard Bochner-like spaces for both $\trial$ and $\test$, cf. Remarks~\ref{rem:resonance-aware-norms} and~\ref{rem:comparison-with-classical-results}. The difference can be explained by noticing that, by an eigenfunction expansion as in Section~\ref{SS:EigenfunctionsExpansion}, we obtain the first-order IVP
\begin{equation}
	\label{eq:heat-weak-reduced}
	u_k' +\lambda_k u_k = f_k
\end{equation}
for $k\geq 1$, whose solution
\begin{equation}
	\label{eq:heat-waek-reduced-solution}
	u_k(t) = e^{-\lambda_k t}\int_0^t f_k(s)e^{\lambda_k s}\mathrm{d}s.
\end{equation}
is in no case resonant. Thus, it is not a surprise that both the approaches discussed in Sections~\ref{SS:FourierExpansion} and~\ref{SS:energy-balance-law} to determine a suitable data norm for the equivalence \eqref{eq:isomorphism} simply recover the norm $\|\cdot\|^2_{\data_k}$ from \eqref{eq:spaces-expanded}.

The Fourier expansion \eqref{eq:fourier-expansion}-\eqref{eq:fourier-expansion-coefficients}, combined with \eqref{eq:heat-waek-reduced-solution}, leads to the identity
\begin{equation*}
	\label{eq:heat-fourier-expansion}
	\frac{1}{\lambda_k}\|u_k'\|^2_{L^2(0,T)} +  \lambda_k \|u_k\|^2_{L^2(0,T)} = \sum_{j=1}^{+\infty} \frac{T}{2\lambda_k} f_{k,j}^2.  
\end{equation*}
The quadratic form on the right-hand side is diagonal, unlike \eqref{eq:data-norm-fourier-expansion}, with coefficients independent of $j$. By Parseval's identity, the right-hand side is nothing but $\lambda_k^{-1}\|f_k\|^2_{L^2(0,T)}$. Therefore, using the Bochner–norm representation \eqref{eq:NormRepresentation}, 
identity \eqref{eq:heat-fourier-expansion} directly yields $\|u\|_\trial = \|f\|_\data$, that is, the operator $\mathcal L = \partial_t + \Delta$ is an isometry from 
$\trial$ onto $\data$.

Replacing \eqref{eq:fcos-fsin} by
\begin{equation*}
	\label{eq:f-exp}
	\fexp(t) := \int_0^t f_k(s)e^{\lambda_k s}\mathrm{d}s,
\end{equation*}
we obtain the point-wise identity
\begin{equation}
	\label{eq:heat-energy-balance-pointwise}
\lambda_k u_k(t)^2 
= 
\lambda_k e^{-2\lambda_k t}\fexp(t)^2, \quad \text{a.e. in $(0,T)$.}
\end{equation}
Unlike \eqref{eq:energy-balance}, the right-hand side is bounded as follows (see below)
\begin{align}
\label{eq:heat-critical-estimate}
	\lambda_k\int_0^Te^{-2\lambda_k t} \fexp(t)^2\dt \le 
	\frac{1}{\lambda_k}\int_0^Tf_k(t)^2\dt
\end{align}
showing that the combination of \eqref{eq:heat-weak-reduced} with \eqref{eq:heat-energy-balance-pointwise} eventually provides the equivalence
\begin{equation*}
\frac{1}{\lambda_k} \|u_k\|^2_{L^2(0,T)} + \lambda_k \|u_k\|^2_{L^2(0,T)} \eqsim
\frac{1}{\lambda_k} \|f_k\|^2_{L^2(0,T)}	
\end{equation*}
with hidden constants independent of $\lambda_k$. Thus, also the the energy balance approach eventually leads to the equivalence $\|u\|_\trial \eqsim \|f\|_\data$. 

\begin{proof}[Proof of \eqref{eq:heat-critical-estimate}]
Note that $\fexp \in H^1(0,T)$ with $\fexp(0) = 0$. Thus, we have
\begin{equation*}
\begin{split}
\lambda_k\int_0^Te^{-2\lambda_k t} \fexp(t)^2 
&= 
2 \lambda_k \int_0^T \fexp(t) (\fexp(t))'\left( \int_t^T e^{-2\lambda_k s} \mathrm{d}s\right)\dt \\
&\leq
\int_0^T |\fexp(t) (\fexp(t))'| e^{-2\lambda_k t}\dt\\
&=
\int_0^T |\fexp(t) f_k(t)| e^{-\lambda_k t}\dt
\end{split}
\end{equation*}
where the last identity follows from the relation $(\fexp(t))' = f_k(t) e^{\lambda_k t}$ a.e. in $(0,T)$. We conclude by invoking the Cauchy-Schwartz inequality.\end{proof}

\subsection{Schrödinger Equation} 
The last example is intended to point out that our results are not restricted to the acoustic wave equation. For this purpose, we consider the IBVP
\begin{equation*}
	\begin{aligned}
		\begin{aligned}
			i\partial_t u - \Delta u & = f && \text{in } \Omega \times (0,T) \\
			u & = 0 && \text{in } \partial\Omega \times (0,T) \\
			u(\cdot, 0) & = 0 && \text{in } \Omega. 
		\end{aligned}
	\end{aligned}
\end{equation*}
By applying the eigenfunction expansion of Section~\ref{SS:EigenfunctionsExpansion}, we are led to an IVP involving the first-order ordinary differential equation
\[
	i u_k' + \lambda_k u = f_k
\]
for $k \geq 1$. Due to the imaginary unit, this problem is much closer to the one arising from the acoustic wave equation \eqref{eq:model-problem-weak-reduced} rather than to the one obtained by the heat equation \eqref{eq:heat-weak-reduced}, because certain source terms give rise to a resonant solution.

Once more, it can be expected that well-posedness \eqref{eq:isomorphism} cannot be established with Bochner-like spaces for both solutions and data. In order to find out a resonance-aware data norm in the spirit of Section~\ref{SS:energy-balance-law}, we observe that the solution can be represented as follows
\begin{align*}
	\Re u_k(t) &= \int_0^t \Re f_k(s)\sin(\sqrt{\lambda_k}(t-s))\mathrm ds + \int_0^t \Im f_k(s)\cos(\sqrt{\lambda_k}(t-s))\mathrm ds, \\
	\Im u_k(t) &= \int_0^t \Im f_k(s)\sin(\sqrt{\lambda_k}(t-s))\mathrm ds - \int_0^t \Re f_k(s)\cos(\sqrt{\lambda_k}(t-s))\mathrm ds,
\end{align*}
with $\Re$ and $\Im$ denoting the real and the imaginary part, respectively. Hence, repeating the argument in Sections~\ref{SS:energy-balance-law} and~\ref{SS:damping} eventually leads to the point-wise energy balance
\[
|u_k(t)|^2 = \left(\Re \fcos(t) + \Im \fsin(t)\right)^2 + \left(\Re \fsin(t) - \Im \fcos(t)\right)^2, \quad \text{a.e. in $(0,T)$}
\]
and to a counterpart of Proposition~\ref{P:data-norm-energy-balance}. Note, in particular, that the right-hand side in this identity equals the one of \eqref{eq:energy-balance} for real-valued source terms. This structural analogy gives hope that an effective technique for the analysis of the acoustic wave equation can be extended also to other wave-like equations.

\bibliographystyle{siam}
\bibliography{bibliography}

@book {Dautray.Lions:92,
    AUTHOR = {Dautray, Robert and Lions, Jacques-Louis},
     TITLE = {Mathematical analysis and numerical methods for science and
              technology. {V}ol. 5},
      NOTE = {Evolution problems. I},
 PUBLISHER = {Springer-Verlag, Berlin},
      YEAR = {1992},
     PAGES = {xiv+709},
}

@book {Dym.McKean:72,
    AUTHOR = {Dym, H. and McKean, H. P.},
     TITLE = {Fourier series and integrals},
    SERIES = {Probability and Mathematical Statistics},
    VOLUME = {No. 14},
 PUBLISHER = {Academic Press, New York-London},
      YEAR = {1972},
     PAGES = {x+295},
}

@book {Ern.Guermond:21-II,
	AUTHOR = {Ern, A. and Guermond, J.-L.},
	TITLE = {Finite elements {II} - {G}alerkin approximation, elliptic and
	mixed {PDE}s},
	SERIES = {Texts in Applied Mathematics},
	VOLUME = {73},
	PUBLISHER = {Springer, Cham},
	YEAR = {[2021] \copyright 2021},
	PAGES = {ix+492}
}

@book {Ern.Guermond:21-III,
	AUTHOR = {Ern, A. and Guermond, J.-L.},
	TITLE = {Finite elements {III} - First-order and time-dependent {PDE}s},
	SERIES = {Texts in Applied Mathematics},
	VOLUME = {74},
	PUBLISHER = {Springer, Cham},
	YEAR = {[2021] \copyright 2021},
	PAGES = {viii+417}
}

@book {Evans:10,
    AUTHOR = {Evans, Lawrence C.},
     TITLE = {Partial differential equations},
    SERIES = {Graduate Studies in Mathematics},
    VOLUME = {19},
   EDITION = {Second},
 PUBLISHER = {American Mathematical Society, Providence, RI},
      YEAR = {2010},
     PAGES = {xxii+749},
}

@article {Feischl:22,
	AUTHOR = {Feischl, M.},
	TITLE = {Inf-sup stability implies quasi-orthogonality},
	JOURNAL = {Math. Comp.},
	FJOURNAL = {Mathematics of Computation},
	VOLUME = {91},
	YEAR = {2022},
	NUMBER = {337},
	PAGES = {2059--2094}
}

@article {Fuhrer.Gonzalez.Karkulik:25,
    AUTHOR = {Führer, Thomas and González, Roberto and Karkulik, Michael},
    TITLE = {Well-posedness of first-order acoustic wave equations and space-time finite element approximation},
    JOURNAL = {IMA Journal of Numerical Analysis},
    PAGES = {drae104},
    YEAR = {2025},
}

@misc {Kreuzer.Zanotti:24a,
	AUTHOR = {Kreuzer, C. and Zanotti, P.},
	TITLE = {Inf-sup theory for the quasi-static {B}iot's equations in poroelasticity},
	HOWPUBLISHED = {arXiv preprint arXiv:2407.02932},
}

@book {Ladyzhenskaya:85,
    AUTHOR = {Ladyzhenskaya, O. A.},
     TITLE = {The boundary value problems of mathematical physics},
    SERIES = {Applied Mathematical Sciences},
    VOLUME = {49},
 PUBLISHER = {Springer-Verlag, New York},
      YEAR = {1985},
     PAGES = {xxx+322},
}

@book {Zank:20,
  AUTHOR={Zank, Marco},
  TITLE={Inf-sup stable space-time methods for time-dependent partial differential equations},
  YEAR={2020},
  PUBLISHER={Verlag der Technischen Universit{\"a}t Graz}
}

@article {Steinbach.Zank:22,
    AUTHOR = {Steinbach, Olaf and Zank, Marco},
     TITLE = {A generalized inf-sup stable variational formulation for the
              wave equation},
   JOURNAL = {J. Math. Anal. Appl.},
  FJOURNAL = {Journal of Mathematical Analysis and Applications},
    VOLUME = {505},
      YEAR = {2022},
    NUMBER = {1},
     PAGES = {Paper No. 125457, 24},
}

@article {Babuska:70,
    AUTHOR = {Babu\v{s}ka, Ivo},
     TITLE = {Error-bounds for finite element method},
   JOURNAL = {Numer. Math.},
  FJOURNAL = {Numerische Mathematik},
    VOLUME = {16},
      YEAR = {1970/71},
     PAGES = {322--333},
}

@article {Brezzi:74,
    AUTHOR = {Brezzi, F.},
     TITLE = {On the existence, uniqueness and approximation of saddle-point
              problems arising from {L}agrangian multipliers},
   JOURNAL = {Rev. Fran\c caise Automat. Informat. Recherche
              Op\'erationnelle S\'er. Rouge},
  FJOURNAL = {Revue Fran\c caise d'Automatique, Informatique et Recherche
              Op\'erationnelle S\'erie Rouge},
    VOLUME = {8},
      YEAR = {1974},
     PAGES = {129--151},
}

@article {Bignardi.Moiola:25,
    AUTHOR = {Bignardi, Paolo and Moiola, Andrea},
     TITLE = {A space-time continuous and coercive formulation for the wave
              equation},
   JOURNAL = {Numer. Math.},
  FJOURNAL = {Numerische Mathematik},
    VOLUME = {157},
      YEAR = {2025},
    NUMBER = {4},
     PAGES = {1211--1258},
}

@article {Henning.Palitta.Simoncini.Urban:22,
    AUTHOR = {Henning, Julian and Palitta, Davide and Simoncini, Valeria and
              Urban, Karsten},
     TITLE = {An ultraweak space-time variational formulation for the wave
              equation: analysis and efficient numerical solution},
   JOURNAL = {ESAIM Math. Model. Numer. Anal.},
  FJOURNAL = {ESAIM. Mathematical Modelling and Numerical Analysis},
    VOLUME = {56},
      YEAR = {2022},
    NUMBER = {4},
     PAGES = {1173--1198},
}

@book {Lions.Magenes:72,
    AUTHOR = {Lions, J.-L. and Magenes, E.},
     TITLE = {Non-homogeneous boundary value problems and applications.
              {V}ol. {I}},
    SERIES = {Die Grundlehren der mathematischen Wissenschaften},
    VOLUME = {Band 181},
      NOTE = {Translated from the French by P. Kenneth},
 PUBLISHER = {Springer-Verlag, New York-Heidelberg},
      YEAR = {1972},
     PAGES = {xvi+357},
}

@article{Ferrari.Perugia:25,
    AUTHOR = {Ferrari, Matteo and Perugia, Ilaria},
	TITLE = {Space-time discretization of the wave equation in a second-order-in-time formulation: a conforming, unconditionally stable method},
	JOURNAL = {arXiv preprint arXiv:2503.11166},
	YEAR = {2025}
}

@article{Ferrari.Perugia.Zampa:25,
    AUTHOR = {Ferrari, Matteo and Perugia, Ilaria and Zampa, Enrico},
	TITLE = {Inf-sup stable space-time discretization of the wave equation based on a first-order-in-time variational formulation},
	JOURNAL = {arXiv preprint arXiv:2506.05886},
	YEAR = {2025},
}

@article {Hiptmair:06,
    AUTHOR = {Hiptmair, R.},
     TITLE = {Operator preconditioning},
   JOURNAL = {Comput. Math. Appl.},
  FJOURNAL = {Computers \& Mathematics with Applications. An International
              Journal},
    VOLUME = {52},
      YEAR = {2006},
    NUMBER = {5},
     PAGES = {699--706},
}

@article {Loscher.Steinbach,
    AUTHOR = {L\"oscher, Richard and Steinbach, Olaf},
     TITLE = {Space-time finite element methods for distributed optimal
              control of the wave equation},
   JOURNAL = {SIAM J. Numer. Anal.},
  FJOURNAL = {SIAM Journal on Numerical Analysis},
    VOLUME = {62},
      YEAR = {2024},
    NUMBER = {1},
     PAGES = {452--475},
}

@article {Mardal.Winther:11,
    AUTHOR = {Mardal, Kent-Andre and Winther, Ragnar},
     TITLE = {Preconditioning discretizations of systems of partial
              differential equations},
   JOURNAL = {Numer. Linear Algebra Appl.},
  FJOURNAL = {Numerical Linear Algebra with Applications},
    VOLUME = {18},
      YEAR = {2011},
    NUMBER = {1},
     PAGES = {1--40},
}

@book {Verfurth:13,
    AUTHOR = {Verf\"urth, R\"udiger},
     TITLE = {A posteriori error estimation techniques for finite element
              methods},
    SERIES = {Numerical Mathematics and Scientific Computation},
 PUBLISHER = {Oxford University Press, Oxford},
      YEAR = {2013},
     PAGES = {xx+393},
}

@article {Hoonhout.Loscher.UrzuaTorres:25,
  AUTHOR={Hoonhout, Daniel and L{\"o}scher, Richard and Urz{\'u}a-Torres, Carolina},
  TITLE={Paving the way to a $\operatorname{T}$-coercive method for the wave equation},
  JOURNAL={arXiv preprint arXiv:2509.02288},
  YEAR={2025}
}

@article {Lischke.Pang.Gulian:20,
    AUTHOR = {Lischke, Anna and Pang, Guofei and Gulian, Mamikon and et al.},
     TITLE = {What is the fractional {L}aplacian? {A} comparative review
              with new results},
   JOURNAL = {J. Comput. Phys.},
  FJOURNAL = {Journal of Computational Physics},
    VOLUME = {404},
      YEAR = {2020},
     PAGES = {109009, 62},
}

@article {Dong.Gergoulis.Mascotto.Wang:25,
  AUTHOR={Dong, Zhaonan and Georgoulis, Emmanuil H and Mascotto, Lorenzo and Wang, Zuodong},
  TITLE={A posteriori error analysis and adaptivity of a space-time finite element method for the wave equation in second order formulation},
  JOURNAL={arXiv preprint arXiv:2509.08537},
  YEAR={2025}
}

@misc{Dong.Mascotto.Wang:25,
      AUTHOR={Zhaonan Dong and Lorenzo Mascotto and Zuodong Wang},
      TITLE={A priori and a posteriori error estimates of a {$\mathcal {C}^0$}-in-time method for the wave equation in second order formulation}, 
      JOURNAL={arXiv preprint 2411.03264},
}
\end{document}